\documentclass[12pt,usenames,dvipsnames]{amsart}
\usepackage{amsmath,amssymb,amsthm}
\usepackage[section]{algorithm}
\usepackage{algorithmic}
\usepackage[inline]{enumitem}
\usepackage{graphicx,setspace}
\usepackage{tikz}
\usetikzlibrary{arrows}
\usepackage{wasysym}
\usepackage[makeroom]{cancel}
\usepackage{fullpage}
\usepackage{comment}
\usepackage{enumitem}
\usepackage{caption}
\usepackage{hyperref}
\usepackage{url}
\allowdisplaybreaks
\usepackage{amsmath,amsfonts,amsthm,verbatim, graphicx,tikz}
\usepackage{amssymb, mathrsfs, centernot}
\usetikzlibrary{shapes.geometric,decorations.pathreplacing,angles,quotes}
\newtheorem{thm}{Theorem}[section]
\newtheorem{lem}[thm]{Lemma}

\newtheorem{conj}{Conjecture}
\newtheorem{cor}[thm]{Corollary}

\newtheorem{ques}[conj]{Question}
\theoremstyle{definition}
\newtheorem{defn}{Definition}[section]

\theoremstyle{remark}
\newtheorem{rmk}[thm]{Remark}

\newcommand{\integers}{\mathbb{Z}}
\newcommand{\ints}{\mathbb{Z}}

\newcommand{\nats}{\mathbb{N}}

\newcommand{\sfs}{\text{sfs}}

\definecolor{lightgray}{rgb}{0.75,0.75,0.75} 
\begin{document}
	\title{On an Algorithm for Comparing the Chromatic Symmetric Functions of Trees}
	
	\author[S. Heil]{S.\ Heil}
	\address{Department of Mathematics\\ 
		Washington University in St. Louis\\ St. Louis, MO 63130, USA}
	\email{\url{sheil@wustl.edu}}
	
	\author[C. Ji]{C. Ji}
	\address{Department of Mathematics\\ 
		Washington University in St. Louis\\ St. Louis, MO 63130, USA}
	\email{\url{caleb.ji@wustl.edu}}
	
	\maketitle
	
	\begin{abstract} 
	    It is a long-standing question of Stanley whether or not the chromatic symmetric function (CSF) distinguishes unrooted trees.  Previously, the best computational result proved that it distinguishes all trees with at most $25$ vertices \cite{Russell}.  In this paper, we present a novel probabilistic algorithm which may be used to check more efficiently that the CSF distinguishes a set of trees.  Applying it, we verify that the CSF distinguishes all trees with up to $29$ vertices.  
	\end{abstract}
	\section{Introduction}
	Richard Stanley asked in \cite{Stanley} whether the chromatic symmetric function (CSF) distinguishes unrooted trees. Since then, it has been proven that the CSF distinguishes all trees in each of several subclasses (\cite{Aliste}, \cite{Martin}, \cite{Orellana}). Tan (\cite{Tan}) and independently Smith, Smith, and Tian (\cite{Smith}) have computationally verified that the CSF distinguishes all trees on at most $23$ vertices, and Russell (\cite{Russell}) has shown computationally that it distinguishes all trees on at most $25$ vertices. 
	
	When expressed with respect to commonly-used bases for the space of symmetric functions, the chromatic symmetric function of an arbitrary tree on $n$ vertices contains a number of distinct terms equal to the number of partitions of $n$, growing super-polynomially with $n$. Therefore, computing the chromatic symmetric function directly requires a super-polynomial number of operations, making verification of Stanley's conjecture for trees on $n$ vertices computationally difficult as $n$ increases. 
	
	We present a probabilistic polynomial time algorithm for determining whether two trees $S$ and $T$ on $n$ vertices have equal chromatic symmetric functions without explicitly computing the chromatic symmetric functions $X_S$ and $X_T$. If in fact $X_S \not= X_T$, with probability at least $1 - \frac{1}{2^k}$ this algorithm returns a proof that $X_S \not= X_T$ in $O(kn^3)$ time, where $k$ is a specified accuracy parameter. Using a variant of this algorithm, we verify computationally that Stanley's chromatic symmetric function distinguishes all trees on at most $29$ vertices up to isomorphism. 
	\section{Preliminaries}\label{sec2}
	We use the notation for symmetric functions found in \cite{Stanley}. 
	\begin{defn}\label{def:pbasis}
		The $i$th \textit{power-sum symmetric function} is defined by \[ p_i(x_1,x_2,\dots) = \sum_{j=1}^\infty x_j^i. \]
		For a partition $\lambda = (\lambda_1, \dots, \lambda_k) \vdash n$, one writes $p_\lambda = \prod_{i=1}^k p_{\lambda_i}$.
	\end{defn}
	\begin{defn}
	    Given a graph $G = (V(G), E(G)),$ a \textit{proper coloring} of $G$ is defined to be a mapping $\kappa : V(G) \to \nats$ such that, for any $u,v \in V(G),$ if $uv \in E(G)$ then $\kappa(u) \not= \kappa(v)$.
	\end{defn}
	\begin{defn}(Stanley, \cite{Stanley}, Definition 2.1)
	    For a graph $G$, Stanley defined the \textit{chromatic symmetric function} of $G$, denoted $X_G$, in \cite{Stanley} as follows: \[ X_G = \sum_{\kappa} \prod_{v \in V(G)} x_{\kappa(v)}, \] where $x_1, x_2, \dots $ are commuting indeterminants and the sum is taken over all proper colorings $\kappa$ of $G$.
	    
	    Additionally, when the graph $G$ is understood, for any proper coloring $\kappa$ we let $x_\kappa = \prod_{v \in V(G)} x_{\kappa(v)}$.
	\end{defn}
	\section{A Probabilistic Algorithm for Distinguishing Chromatic Symmetric Functions}\label{sec3}
	\begin{defn}Let $G$ be a graph on $n$ vertices, and let $S_{v,c}$ be the set of all proper colorings of $G$ such that the color of vertex $v$ is $c$. Then, for each vertex $v$ of $G$ and each color $c$, we define the function $Z_G^v(c)$ so that $$Z_G^v(c) = \sum_{\kappa \in S_{v,c}} \prod_{u \in G} x_{\kappa(u)}.$$ We will let $Y_G^v(c) = X_G-Z_G^v(c)$ be the sum of all terms $x_\kappa = \prod_{u \in G} x_{\kappa(u)}$ in the CSF $X_G$ such that $\kappa(v) \not= c$. When the vertex $v$ is implied by context (such as when $G$ is a rooted tree and $v$ is the root of $G$), we will simply write $Z_G(c)$ and $Y_G(c)$ for $Z_G^v(c)$ and $Y_G^v(c)$, respectively.
	\end{defn}
	The following lemma follows directly from the definitions.
	\begin{lem}\label{lem:csf-Zsplit}
		Let $G$ be a graph on the vertex set $\{ 1, 2, \dots, n\}$, and let $v$ be a vertex of $G$. Then, we have $$\sum_{c=1}^\infty Z_G^v(c) = X_G.$$
	\end{lem}
	The function $Z_G^v(c)$ simplifies the task of finding the chromatic symmetric function $X_G$ for a graph $G$ by reducing it to cases in which the color of a certain vertex of the graph is fixed. This is particularly helpful for trees, since we can reconstruct $Z_G(c)$ for a tree given the corresponding information about its subtrees.
	
	\begin{lem}\label{lem:csf-decomp}
		Let $T$ be a tree rooted at vertex $v$, and let $v_1$ through $v_k$ be the vertices of $T$ adjacent to $v$. For each $1 \le i \le k$, let $T_i$ be the connected component of $T \setminus \{\{v, v_i\}\}$ containing $v_i$. Then $$Z_T(c) = x_c \prod_{i=1}^k (X_{T_i} - Z_{T_i}(c)) = x_c\prod_{i=1}^k Y_{T_i}(c).$$
	\end{lem}
	\begin{proof}
	    If vertex $v$ has color $c$ in a proper coloring of $T$, then each of the connected components $T_i$ can be viewed as a subtree with root $v_i$.  In each of these proper colorings, $v_i$ cannot have color $c$, so every valid coloring where vertex $v$ has color $c$ is contained in the set of colorings implied by the monomials in the expansion of $x_c \prod_{i=1}^k (X_{T_i} - Z_{T_i}(c))$.  We claim that all of these monomials do in fact correspond to valid colorings.  Indeed, it is clear that none of the edges contained within any of the $T_i$ belong to vertices of the same color.  Every other edge is from the root $v$ to some $v_i$, and we know that none of the $v_i$ are the same color as $c$.  Thus the claim is true, so $Z_T(c) = x_c \prod_{i=1}^k (X_{T_i} - Z_{T_i}(c)) = x_c\prod_{i=1}^k Y_{T_i}(c)$ as desired.

	\end{proof}
	
	\begin{lem}\label{lem:csf-colorsplit}
		Let $T$ be a tree on $n$ vertices rooted at $v$. Then, there exists a unique $n$-tuple of symmetric functions, $(F_1, \dots, F_n)$, each in the indeterminants $x_1, x_2, \dots$, such that, for any $c\in \nats$, \[ Z_T(c) = \sum_{i=1}^n x_c^i F_i. \]
	\end{lem}
	\begin{proof}
		First, we will prove that such an $n$-tuple of symmetric functions must exist. 
		
		We proceed by induction on $n$. 
		
		For the base case, $n=1$, $T$ must contain only the single vertex $v$, so there is a single coloring $\kappa$ of the vertices of $T$ such that $\kappa(v) = c$. Therefore, we have that $X_T = \sum_{i=1}^\infty x_i$, and $Z_T(c) = x_c$. Thus, for the constant symmetric function $F_1(x_1,x_2,\dots) = 1$, we have that $Z_T(c) = x_cF_1 = \sum_{i=1}^1 x_c^iF_i$ for all $c \in \nats$, and the base case is proven. 
		
		Next, in the inductive case, assume that $n \ge 2$ and the lemma holds for all trees of size at most $n-1$. Then, let $v_1, v_2, \dots, v_k$ be the vertices of $T$ adjacent to the root vertex $v$, and for each $1 \le i \le k$ let $T_i$ be the connected component of $T \setminus \{ \{ v, v_i \} \}$ rooted at $v_i$ and let $n_i$ be the number of vertices in  $T_i$ (including the root $v_i$). By our inductive assumption, for each $1 \le i \le k$ there exists a sequence of symmetric functions $F_{i,1}, F_{i,2}, \dots, F_{i,n_i}$ such that for all $c \in \nats$, $$Z_{T_i}(c) = \sum_{j=1}^{n_i} x_c^j F_{i,j}.$$ By Lemma \ref{lem:csf-decomp}, we have that $$Z_T(c) = x_c \prod_{i=1}^k (X_{T_i} - Z_{T_i}(c)) = x_c \prod_{i=1}^k \left( X_{T_i} - \sum_{j=1}^{n_i} x_c^j F_{i,j}\right) = x_c \prod_{i=1}^k \sum_{j=0}^{n_i} x_c^j G_{i,j},$$ where we let $G_{i,0} = X_{T_i}$ and $G_{i,j} = -F_{i,j}$ for $1 \le j \le n_i$.
		
		For each nonnegative integer $i$, let $P_{i,k}$ be the set of all ordered $k$-tuples of nonnegative integers $p = (p_1, \dots, p_k)$ such that $p_1+\dots+p_k = i$ and $p_j \le n_j$ for each $1 \le j \le k$. (Note that we must have $i \le \sum_{j=1}^k n_j = n-1$.) Then, we can rewrite our product expression for $Z_T(c)$ as follows:
		
		\begin{align*}Z_T(c)&= x_c \prod_{i=1}^k \sum_{j=0}^{n_i} x_c^j G_{i,j}\\ &= x_c \sum_{i=0}^{n-1} \sum_{p \in P_{i,k}} \prod_{j=1}^k x_c^{p_j} G_{j,p_j} \\&= x_c \sum_{i=0}^{n-1} x_c^i \sum_{p \in P_{i,k}} \prod_{j=1}^k G_{j,p_j} \\&= \sum_{i=1}^n x_c^i \left( \sum_{p \in P_{i-1,k}} \prod_{j=1}^k G_{j,p_j} \right)\end{align*}
		
		Now, for each $1 \le i \le n$, let $F_i = \sum_{p \in P_{i-1,k}} \prod_{j=1}^k G_{j,p_j}$. Then, we have that $F_i$ is a symmetric function for each $1 \le i \le n$ and is independent of $c$, as desired; hence the inductive step is complete.
		
		Therefore, by induction, we conclude that for any positive integer $n$, any rooted tree $T$ on $n$ vertices, and any positive integer $c$, there exist symmetric functions $F_1, F_2, \dots, F_n$ such that $ Z_T(c) = \sum_{i=1}^n x_c^i F_i$.
		
		Next, we will prove there is a \textit{unique} $n$-tuple of symmetric functions $(F_1, F_2, \dots, F_n)$ satisfying $ Z_T(c) = \sum_{i=1}^n x_c^i F_i$.
		
		Let $(F_1, F_2, \dots, F_n)$ and $(G_1, G_2, \dots, G_n)$ be $n$-tuples of symmetric functions satisfying $$\sum_{i=1}^n x_c^i F_i = \sum_{i=1}^n x_c^i G_i = Z_T(c).$$ We claim that $F_i = G_i$ for all $1 \le i \le n$; assume by way of contradiction that $F_i \not= G_i$ for some $i$. Then, subtracting the second sum from the first, we obtain the equation \[ \sum_{i=1}^n x_c^i (F_i-G_i) = 0. \] Let $j$ be the minimal index for which $F_j \not= G_j$. Then, we have that $ \sum_{i=j}^n x_c^i (F_i-G_i) = 0$, so $x_c^j (F_j-G_j) = \sum_{i=j+1}^n x_c^i (F_i-G_i)$.
		
		If $j < n$, then we have $x_c^j(F_j-G_j) = x_c^{j+1} \sum_{i=1}^{n-j} (F_{i+j}-G_{i+j})$, so $F_j-G_j = x_c \sum_{i=1}^{n-j} (F_{i+j}-G_{i+j})$. But since $F_j - G_j$ is a nonzero symmetric function, it cannot be divisible by $x_c$, so this is a contradiction. 
		
		If $j = n$, then we have $x_c^n (F_n-G_n) = 0 \implies F_n = G_n$, which is also a contradiction since we assumed $F_j \not= G_j$.  
		
		Therefore, in either case we obtain a contradiction, so we must have had that $F_i = G_i$ for all $1 \le i \le n$. Thus, the $n$-tuple of functions $(F_1, F_2, \dots, F_n)$ such that $Z_T(c) = \sum_{i=1}^n x_c^i F_i$ must be unique.
	\end{proof}
	
	\begin{defn}\label{def:csf-seq}
		For a tree $T$ on $n$ vertices and a vertex $v$ of $T$, define the \textit{symmetric function sequence of }$T$\textit{ rooted at }$v$ to be the unique sequence $\sfs(T,v) = (F_1, F_2, \dots, F_n)$ of symmetric functions  satisfying $X_T^v(c) = \sum_{i=1}^n x_c^i F_i$. By Lemma \ref{lem:csf-colorsplit}, $\sfs(T,v)$ exists and is well-defined for any tree $T$ and vertex $v$ of $T$.
	\end{defn}
	
	\begin{cor}\label{cor:csf-psplit}
	    If, for some vertex $v$ of $T$, $\sfs(T,v) = (F_1, \dots, F_n)$, then \[ X_T = \sum_{i=1}^n p_iF_i. \]
	\end{cor}
	\begin{proof}
		Let $v$ be the root of $T$. Then, by Lemma \ref{lem:csf-Zsplit}, we have that \[ X_T = \sum_{c=1}^\infty Z_T^v(c). \] Furthermore, by Lemma \ref{lem:csf-colorsplit} it follows that \[ Z_T^v(c) = \sum_{i=1}^n x_c^i F_i\] for $\sfs(T,v) = (F_1, \dots, F_n)$. By Definition \ref{def:pbasis}, the coefficients in the sum over all colors $c$ of $Z_T^v(c)$ are exactly the $p$-basis elements $p_1, p_2, \dots, p_n$, so we have \[ X_T = \sum_{c=1}^\infty \sum_{i=1}^n x_c^i F_i = \sum_{i=1}^n \sum_{c=1}^\infty x_c^i F_i = \sum_{i=1}^n p_iF_i, \] as desired.
	\end{proof}
	
	We now present a recursive algorithm for computing the chromatic symmetric function $X_T$ of a tree $T$ in the $p$-basis, using Lemma \ref{lem:csf-decomp} and Corollary \ref{cor:csf-psplit}; its essential procedure is computing the symmetric function sequence.
	\begin{algorithm}
		\caption{COMPUTE-CSF($n$-vertex tree $T$)}\label{alg:computecsf}
		\begin{algorithmic}
			\item[] $v \leftarrow $ arbitrary vertex of $T$
			\item[] $(F_1, F_2, \dots, F_n) \leftarrow$ COMPUTE-SFS$(T, v)$
			\RETURN $p_1F_1+p_2F_2+\dots+p_nF_n$
		\end{algorithmic}
	\end{algorithm}
	\begin{algorithm}\label{alg:computesfs}
		\caption{COMPUTE-SFS($n$-vertex tree $T$, vertex $v$)}
	\end{algorithm}
	\begin{algorithmic}
		\IF{$T$ is single vertex}
		\RETURN ($1$)
		\ENDIF
		\item[] $v_1, \dots, v_k \leftarrow$ vertices adjacent to $v$ in $T$
		\FOR{$i = 1$ to $k$}
		\item[] $T_i \leftarrow$ connected component of $T$ rooted at $v_i$
		\ENDFOR
		\item[] $F_1 \leftarrow 1$
		\FOR{$i=2$ to $n$}
		\item[] $F_i \leftarrow 0$
		\ENDFOR
		\item[] $d \leftarrow 1$
		\FOR{$i = 1$ to $k$}
		\FOR{$j=1$ to $d$}
		\item[] $H_i \leftarrow F_i$
		\ENDFOR
		\item[] $G_1, \dots, G_m \leftarrow$ COMPUTE-SFS$(T_i, v_i)$
		\item[] $G_0 \leftarrow p_1G_1+p_2G_2+\dots+p_mG_m$ // compute the CSF $X_{T_i}$
		\FOR{$j=1$ to $m$}
		\item[] $G_i \leftarrow -G_i$
		\ENDFOR
		\FOR{$j=1$ to $m+d$}
		\item[] $F_i \leftarrow 0$
		\ENDFOR
		\FOR{$j=1$ to $d$}
		\FOR{$p=0$ to $m$}
		\item[] $F_{j+p} \leftarrow F_{j+p} + H_jG_p$
		\ENDFOR
		\ENDFOR
		\item[] $d \leftarrow d+m$
		\ENDFOR
		\RETURN $(F_1, F_2, \dots, F_n)$
	\end{algorithmic}

	Note that a call to COMPUTE-SFS($T, v$) will recursively result in calls to COMPUTE-SFS for the subtree of $T$ rooted at each vertex $u$, for a total of $n$ function calls. After each COMPUTE-SFS call on a subtree of $m$ vertices, there are $m$ symmetric function multiplications and $m-1$ additions, followed by $(m+1)d$ multiplications and additions, for a total of at most $(m+1)(d+1)$ of each. Since $m \le n$ and $d \le n$, and there are $n$ COMPUTE-SFS calls, the number of symmetric function multiplications and additions required for COMPUTE-SFS($T,v$) is bounded by a polynomial in $n$ for a tree $T$ of size $n$. 
	
	The drawback to this recursive algorithm is the high computational cost of each symmetric function multiplication and addition. Since the chromatic symmetric function $X_T$ of $T$, if represented in the $p$-basis, can in the worst case contain a term for each partition of $n$, the cost of each symmetric function multiplication and addition grows proportionally to $e^{O(\sqrt n)}$.
	
	To efficiently determine that the chromatic symmetric functions of a set of trees are distinct without incurring the super-polynomial cost of explicitly computing the complete chromatic symmetric function of each tree, we will define a homomorphism from the set of chromatic symmetric functions to a smaller finite set.
	
	It follows from  Theorem 2.5 of \cite{Stanley} that for any chromatic symmetric function $X_G$, there is some polynomial $F \in \ints[p_1, p_2, \dots]$ such that $X_G = F(p_1, p_2, \dots)$. An immediate corollary is that any linear combination $X = k_1X_{G_1} + \dots + k_nX_{G_n}$ of CSFs, for integers $k_1, \dots, k_n$, is an element of $\ints[p_1, p_2, \dots]$.
	
	\begin{defn}\label{def:eval-sf}
	Let $X \in \ints[p_1,p_2,\dots]$. Then, for each $q \in \nats$, and each infinite tuple $C = (C_1, C_2, \dots) \in (\ints / q\ints)^\infty$, define the $C$\textit{-evaluation modulo }$q$ of $X$ to be the image of $F$ under the evaluation homomorphism $\pi_{q,C} : \ints[p_1, p_2, \dots] \to \ints / q \ints$ satisfying $\pi_{q,C}(F) = F(C_1, C_2, \dots)$. We denote the $C$-evaluation modulo $q$ of $X$ by $\varphi_{q,C}(X)$.
	\end{defn}
	
	\begin{lem}\label{lem:homom-ver}
		For each $q \in \nats$, $C \in (\ints / q\ints)^\infty$, we have that $\varphi_{q,C}$ is a homomorphism from the polynomial ring $\ints[p_1,p_2,\dots]$ to the ring $\ints / q \ints$.
	\end{lem}
	\begin{proof}
		Let $\sigma$ be the trivial isomorphism from $\ints[p_1,p_2,\dots]$ (with the power-sum symmetric functions) to $\ints[y_1,y_2,\dots]$ (with arbitrary variables $y_i$) defined by $\sigma(f(p_1,p_2,\dots)) = f(y_1,y_2,\dots)$. Now let $\pi : \ints[y_1, y_2, \dots] \to \ints / q \ints$ be the evaluation homomorphism taking $F$ to $F(C_1, C_2, \dots)$. Then, since $\sigma: \ints[p_1,p_2,\dots] \to \ints[y_1,y_2,\dots]$ and $\pi : \ints[y_1,y_2,\dots] \to \ints / q\ints$ are both homomorphisms, their composition $\pi \circ \sigma : \ints[p_1, p_2, \dots] \to (\ints / q\ints)$ is also a homomorphism, as desired.
	\end{proof}
	
	Using the $C$-evaluation modulo $q$ of the chromatic symmetric function and the fact that $\varphi_{q,C}$ is a homomorphism, we provide an analogous version of Algorithm \ref{alg:computecsf} to compute $\varphi_{q,C}(X_T)$ for a tree $T$ without fully computing $X_T$. 
	
		\begin{algorithm}
		\caption{COMPUTE-EVAL-CSF($n$-vertex tree $T$, $q \in \nats$, $C \in (\ints / q \ints)^n)$}\label{alg:computeecsf}
		\begin{algorithmic}
			\item[] $v \leftarrow $ arbitrary vertex of $T$
			\item[] $(r_1, r_2, \dots, r_n) \leftarrow$ COMPUTE-EVAL-SFS$(T, v, q, C)$
			\RETURN $C_1r_1+C_2r_2+\dots+C_nr_n \mod q$
		\end{algorithmic}
	\end{algorithm}
	\begin{algorithm}
		\caption{COMPUTE-EVAL-SFS($n$-vertex tree $T$, vertex $v$, $q \in \nats$, $C \in (\ints / q \ints)^n$)}\label{alg:computeesfs}
	\end{algorithm}
	\begin{algorithmic}
		\IF{$T$ is single vertex}
		\RETURN ($1$)
		\ENDIF
		\item[] $v_1, \dots, v_k \leftarrow$ vertices adjacent to $v$ in $T$
		\FOR{$i = 1$ to $k$}
		\item[] $T_i \leftarrow$ subtree of $T$ rooted at $v_i$
		\ENDFOR
		\item[] $r_1 \leftarrow 1$
		\FOR{$i=2$ to $n$}
		\item[] $r_i \leftarrow 0$
		\ENDFOR
		\item[] $d \leftarrow 1$
		\FOR{$i = 1$ to $k$}
		\FOR{$j=1$ to $d$}
		\item[] $t_i \leftarrow r_i$
		\ENDFOR
		\item[] $s_1, \dots, s_m \leftarrow$ COMPUTE-EVAL-SFS$(T_i, v_i, q, C)$
		\item[] $s_0 \leftarrow C_1s_1+C_2s_2+\dots+C_ms_m \mod q$ // compute $\varphi_{q,C}(X_{T_i})$
		\FOR{$j=1$ to $m$}
		\item[] $s_i \leftarrow -s_i$
		\ENDFOR
		\FOR{$j=1$ to $m+d$}
		\item[] $r_i \leftarrow 0$
		\ENDFOR
		\FOR{$j=1$ to $d$}
		\FOR{$p=0$ to $m$}
		\item[] $r_{j+p} \leftarrow r_{j+p} + t_js_p \mod q$
		\ENDFOR
		\ENDFOR
		\item[] $d \leftarrow d+m$
		\ENDFOR
		\RETURN $(r_1, r_2, \dots, r_n)$
	\end{algorithmic}

	\begin{lem}\label{lem:algquadratic}
		For a tree $T$ on $n$ vertices and a modulus $q$, Algorithm \ref{alg:computeecsf} terminates in $O(n^2 (\log q)^2)$ time.
	\end{lem}

	\begin{proof}
	    First, the additional computation in algorithm \ref{alg:computeecsf} after the call to \ref{alg:computeesfs} includes $n$ multiplications and additions of $q$-bit integers, which takes $O(n (\log q)^2)$ time, so it suffices to show that \ref{alg:computeesfs} terminates in $O(n^2(\log q)^2)$ time.
	    
	    Thus, we claim that for each positive integer $n$ and for any tree $T$ on $n$ vertices, Algorithm \ref{alg:computeesfs} requires at most $12n^2$ addition, multiplication, and modulus operations on elements of $\ints / q\ints$. 
	    
	    We proceed by induction. When $n = 1$, the algorithm immediately terminates and returns $(1)$, so the base case holds.
	    
	    We now assume inductively that our claim holds for all trees on at most $n-1$ vertices, and prove that it must also hold for trees on $n$ vertices.
	    
	    Let $T$ be a tree on $n$ vertices, select an arbitrary root vertex $v$, and let $T_1, \dots, T_k$ be the subtrees of $T$ rooted at the children of $v$. Let $m_1, \dots, m_k$ be the numbers of vertices in $T_1, \dots, T_k$, respectively. Also, let $m_0 = 1$ for simplicity. 
	    
	    For each integer $1 \le i \le k$, let $d_i = \sum_{j=0}^{i-1} m_i$. Note that since initially $d = 1$ and $m_i$ is added to $d$ after iteration $i$, we have that $d_i$ is the initial value of $d$ during the iteration of the outer loop corresponding to subtree $T_i$.
	    
	    On iteration $i$, the initial for loop requires $d_i$ operations. Then, by our inductive assumption, the call to COMPUTE-EVAL-SFS requires at most $12m_i^2$ operations. The following line includes at most $m_i$ each of addition, multiplication, and modulus operation, for a total of $3m_i$ operations. The for loop initializing the $r_i$ values to zero requires $d_i+m_i$ operations. The nested for loops in which $r_js_p$ is added to $r_{j+p}$ require at most $3d_i(m_i+1)$ operations, since one addition, one multiplication, and one modulus operation takes place in the inner loop.
	    
	    Therefore, the total number of operations performed on iteration $i$ of the loop is at most $d_i+12m_i^2+3m_i+d_i+m_i+3d_i(m_i+1) = 12m_i^2+3m_id_i+5d_i+4m_i \le 12m_i^2 + 12m_id_i$. 
	    
	    The number of additional steps performed in the outer loop is at most $5n$, including $n$ each for finding the vertices adjacent to $v$, finding the subtrees of $T$ rooted at these vertices, initializing the $r_i$ values, and returning the final sequence, and the additional constant-time operations.
	    
	    Taking the sum over all iterations and adding in the operations from the outer loop, we have that the total number of operations required is at most \begin{align*} 5n+ \sum_{i=1}^k 12m_i^2 + 12m_id_i &\le 5n+\sum_{i=1}^k 12 m_i^2 + \sum_{i=1}^k 12 m_i + \sum_{i=1}^k \sum_{j=1}^{i-1} 12 m_im_j \\
	    &\le 5n+\sum_{i=1}^k 12 m_i^2 + 12m_2m_1+ \sum_{i=2}^{k-1} 12 m_km_i + 12m_km_1+ \sum_{i=1}^k \sum_{j=1}^{i-1} 12 m_im_j \\
	    &\le 5n+\sum_{i=1}^k 12m_i^2 + \sum_{i=1}^k \sum_{j=i+1}^k 12 m_im_j + \sum_{i=1}^k \sum_{j=1}^{i-1} 12 m_im_j \\
	    &\le 5n + \sum_{i=1}^k \sum_{j=1}^k 12m_im_j \\
	    &\le 5n+12\left( \sum_{i=1}^k m_i \right)^2 \\
	    &= 5n+12(n-1)^2 \\
	    &= 5n+12n^2-24n+12 \\
	    &\le 12n^2.
	    \end{align*}
	
	    Thus, by induction, our claim is proven.
	    
	    Finally, since all addition, multiplication, and modulus operations are performed on positive integers at most $q$, the time per operation is $O((\log q)^2)$. Therefore, as there are at most $12n^2$ operations, the total runtime of Algorithm \ref{alg:computeesfs} is $O(n^2 (\log q)^2)$, which implies that Algorithm \ref{alg:computeecsf} terminates in $O(n^2 (\log q)^2)$ time, as desired.
	   
	\end{proof}

	As Algorithm \ref{alg:computeecsf} can be performed using a number of multiplications and additions of elements of $\ints / q \ints$ that is polynomial in $n$, if $q = O(\exp(p(n)))$ for some polynomial $p$, then this algorithm will terminate in polynomial time. By Remark \ref{rmk:contra}, to show that $X_S \not= X_T$ it suffices to find such a modulus $q$ and an $n$-tuple $C \in (\ints / q \ints)^n$ such that $\varphi_{q,C}(X_S) \not= \varphi_{q,C}(X_T)$. This leads us presently to our main theorem, after one final lemma.
	
	\begin{lem}\label{lem:probbound}
		Let $q$ be a prime, and let $f \in (\ints /q \ints)[x_1, x_2, \dots, x_m]$ be a polynomial of degree $n$ that is not identically zero (e.g. there exists a coefficient of $f$ that is not divisible by $q$). Then, if $C_1, C_2, \dots, C_m$ are randomly chosen elements of $\ints / q \ints$, the probability that $f(C_1, C_2, \dots, C_m) \equiv 0 \pmod q$ is at most $\frac{n}{q}$. 
	\end{lem}
	
	\begin{proof}
		In this lemma, we will use the notation $\Pr[X] = p$ to denote that the probability of event $X$ occurring is $p$. The claim is true when $m=1$. We proceed by induction on $m$.

		Next, for the inductive step, we assume that our claim holds for polynomials in at most $m-1$ variables, for some $m \ge 2$, and we will prove that it also holds for $m$-variable polynomials. We group the terms of the polynomial $f$ by powers of $x_m$: for some polynomials $g_0, g_1, \dots, g_n \in (\ints / q \ints)[x_1,x_2, \dots, x_{m-1}]$, we have that \[ f(x_1, x_2, \dots, x_m) = \sum_{i=0}^a g_i(x_1, x_2, \dots, x_{m-1})x_m^i \]
		for some $1\le a < n$.
		
		There are two disjoint cases: either all the $g_i(x_1, x_2, \dots, x_{m-1})$ are identically zero, or there is some $i$ such that $g_i(x_1, \dots, x_{m-1})$ is not identically zero. Since the degree of $f$ is $n$, the degree of $g_a$ is at most $n-a$, so the probability that $g_a\equiv 0\pmod q$ is at most $\frac{n-a}{q}$.  On the other hand, if not all the  $g_i(x_1, x_2, \dots, x_{m-1})$ are $0$, then by the inductive hypothesis, we have that $\Pr[f(C_1, C_2, \dots, C_m) \equiv 0\pmod q] \le \frac{a}{n}$.  Since the probability that $f(C_1, C_2, \dots, C_m)$ is $0$ is at most the probability that every $g_i(x_1, x_2, \dots, x_{m-1})$ is identically $0$ added to the probability that they are not all $0$ multiplied by the probability that $f(C_1, C_2, \dots, C_m)$ is $0$ in that case, we have $\Pr[f(C_1, C_2, \dots, C_m) \equiv 0 \pmod q]\le \frac{n-a}{q} + 1\cdot\frac{a}{q}=\frac{n}{q}$. 
		
		Thus, by induction, our claim holds for all positive integers $m$ and each value of $n$.
	\end{proof}
	\begin{algorithm}
		\caption{SHOW-DISTINCT-CSFS($n$-vertex tree $S$, $n$-vertex tree $T$, $k \in \nats$)}\label{alg:showdistinctcsfs}
		\begin{algorithmic}
			\item[] $q \leftarrow n^2$
			\item[] $primeCount \leftarrow 0$
			\WHILE{$primeCount < n$}
			\IF{$q$ is determined to be prime by trial division}
			\FOR{$i=1$ to $\lfloor k/\log_2(n) \rfloor$}
			\FOR{$j=1$ to $n$}
			\item[] $C_j \leftarrow $ random element of $\ints / q \ints$
			\ENDFOR
			\item[] $r_S \leftarrow $COMPUTE-EVAL-CSF$(S, q, C)$
			\item[] $r_T \leftarrow $COMPUTE-EVAL-CSF$(T, q, C)$
			\IF{$r_S \not= r_T$}
			\RETURN `Proved that $X_S \not= X_T$.'
			\ENDIF
			\ENDFOR
			\item[] $primeCount \leftarrow primeCount+1$
			\ENDIF
			\item[] $q \leftarrow q+1$
			\ENDWHILE
			\RETURN `Could not prove that $X_S \not= X_T$.'
		\end{algorithmic}
	\end{algorithm}
	
	\begin{thm}
		For trees $S$ and $T$ on $n$ vertices such that $X_S \not= X_T$ and for each $k \in \nats$, with probability at least $1 - 2^{-k}$ Algorithm \ref{alg:showdistinctcsfs} will prove that $X_S \not= X_T$ by generating a positive integer $q$ and a $n$-tuple $C = (C_1, \dots, C_n) \in (\ints / q \ints)^n$ such that $\varphi_{q,C}(X_S) \not= \varphi_{q,C}(X_T)$ in $O(n^3k)$ time. 
	\end{thm}
	\begin{proof}
		First, we prove that if $X_S$ and $X_T$ are distinct, Algorithm \ref{alg:showdistinctcsfs} will obtain a verification that $X_S \not= X_T$ with probability at least $1-2^{-k}$. 
		
		Let $f_S$ and $f_T$ be polynomials such that $f_S(p_1, \dots, p_n) = X_S$ and $f_T(p_1, \dots, p_n) = X_T$, where $p_1, \dots, p_n$ are elements of the $p$-basis for symmetric functions, and let $f = f_S-f_T$. By Theorem $2.5$ of \cite{Stanley}, we have that $X_T = \sum_{U \subseteq E(T)} (-1)^{|U|} p_{\lambda(U)}$, where $E(T)$ is the set of edges of $T$. Since $T$ is a tree on $n$ vertices, we have that $|E(T)| = n-1$, so $|\mathcal P(E(T))| = 2^{n-1}$. Therefore, the sum of the absolute values of the coefficients of $X_T$ in the $p$-basis is at most $2^{n-1}$ and the same result holds for $X_S$. Therefore, by the Triangle Inequality, the sum of the absolute values of the coefficients of $X_S-X_T$ in the $p$-basis is at most $2^n$. 
		
		As $f(p_1,p_2,\dots,p_n) = X_S-X_T$, this implies that the sum of the absolute values of the coefficients of $f$ is at most $2^n$, so each coefficient of $f$ has absolute value bounded by $2^n$.  As Algorithm \ref{alg:showdistinctcsfs} generates $n$ distinct primes larger than $n^2$, it takes at most $\log_{n^2}{2^n} = O(\frac{n}{\log n})$ primes till their product is greater than $2^n$.  Then one of them, say $q$, cannot divide all the coefficients of $f$.  Hence, for this prime $q$ we have that $f(x_1, x_2, \ldots, x_n)$ is not identically zero over $\ints /q\ints$.
		The algorithm generates an $n$-tuple $C$ of randomly-selected residues modulo $\ints / q \ints$ and then computes $\varphi_{q,C}(X_S)$ and $\varphi_{q,C}(X_T)$. 
		
		By Lemma \ref{lem:probbound}, since $f$ is a polynomial in $n$ variables, $\deg f \le n$, and $f$ is not identically $0$ over $\ints / q \ints $ for at least one choice of $q$, the probability that $f(C_1, C_2, \dots, C_n) \equiv 0 \pmod q$ is at most $\frac{n}{q} < \frac{n}{n^2} = \frac{1}{n}$. Therefore, with probability at least $1-\frac{1}{n}$, we have that $$f(C_1, C_2, \dots, C_n) = \varphi_{q,C}(X_S-X_T) = \varphi_{q,C}(X_S)-\varphi_{q,C}(X_T) \not= 0,$$ in which case the algorithm has shown that $\varphi_{q,C}(X_S) \not= \varphi_{q,C}(X_T)$ and hence returns that $X_S \not= X_T$. 
		
		The algorithm generates $k$ independent $n$-tuples $C$, each leading to a proof that $X_S \not= X_T$ with probability at least $1-\frac 1n$, so the probability that it does not return $X_S \not= X_T$ after $\frac{k}{\log_2 {n}}$ iterations is at most $\frac{1}{n^\frac{k}{\log_2 {n}}}=\frac{1}{2^k}$, so it takes $O(\frac{k}{\log_2 {n}})$ iterations to achieve this desired probability. Hence, if $X_S \not= X_T$, then the algorithm will find a pair $(q,C)$ for which $\varphi_{q,C}(X_S) \not= \varphi_{q,C}(X_T)$, showing that $X_S \not= X_T$, with probability at least $1-2^{-k}$.
		
		Next, we will show that Algorithm \ref{alg:showdistinctcsfs} terminates after $O(n^3k)$ steps.
		
		It was proven in \cite{Rosser} that $\frac{x}{\log x + 2} \le \pi(x) \le \frac{x}{\log x - 4}$, where $\pi(x)$ is the number of primes less than $x$. Applying these bounds, we claim that for sufficiently large $n$ there are at least $\frac{n}{\log n}$ primes between $n^2$ and $2n^2$. We have that 
		\begin{align*}
	           \pi(2n^2)-\pi(n^2) &\ge \frac{2n^2}{2\log n + \log 2 + 2} - \frac{n^2}{2 \log n - 4} \\
	           &\ge \frac{2n^2}{2 \log n + 4} - \frac{n^2}{2 \log n - 4} \\
	           &\ge \frac{2n^2(2 \log n - 4)-n^2(2 \log n + 4)}{4(\log n)^2 - 16} \\
	           &\ge \frac{2n^2 \log n - 12n^2}{4 (\log n)^2 - 16} \\
	           &\ge \frac{n^2 (\log n - 6)}{2(\log n)^2} \\
	           &\ge \frac{n^2 }{2 (\log n)^2} \\
	           &> \frac{n}{\log n}
		\end{align*}
		for all $n > e^7$, as desired. Therefore, for sufficiently large $n$ we must test at most $n^2$ integers for primality, each of which is at most $2n^2$. As trial division can determine whether or not $n$ is prime in $O(\sqrt n)$ time, this implies that our algorithm will generate the desired $\frac{n}{\log n}$ primes in $O(n^3)$ time. 
		
		For each prime $q$ generated by our algorithm, it first generates $n$ random elements of $\ints / q \ints$, taking $O(n \log q) = O(n \log n)$ time (since $q < 2n^2 < n^3$, $\log q \le 3 \log n$). Then, there are $2$ calls to COMPUTE-EVAL-CSF on trees of size $n$ using the modulus $q$. By Lemma \ref{lem:algquadratic}, these calls require $O(n^2(\log q)^2) = O(n^2(\log n)^2)$. Thus, each iteration of the inner loop requires $O(n \log n + n^2 (\log n)^2) = O(n^2 (\log n)^2)$ time.
		
		As there are $O(\frac{k}{\log n})$ iterations for each of $O(\frac{n}{\log n})$ primes, the total runtime from the inner-loop computations (everything excluding the primality testing) is $O(n^2 (\log n)^2 \cdot \frac{k}{\log n} \cdot \frac{n}{\log n}) = O(n^3k)$.
		
		Therefore, the total runtime of Algorithm \ref{alg:showdistinctcsfs} is $O(n^3+n^3k) = O(n^3k)$.
		\end{proof}
	
	\section{Computational Results and an Additional CSF Conjecture}\label{sec4}
	We now apply a variant of Algorithm \ref{alg:computeecsf} to extend Russell's computational result in \cite{Russell} that the chromatic symmetric function distinguishes all unrooted trees on at most $25$ vertices up to trees on at most $29$ vertices.
	
	We begin by formulating a stronger form of Stanley's conjecture that the chromatic symmetric function distinguishes unrooted trees. This stronger conjecture implies the original conjecture and is less computationally expensive to verify.
	
	Let $T = (V(T), E(T))$ be a tree on $n$ vertices. As in \cite{Stanley}, for each subset $S \subset E(T)$ define $\lambda(S)$ to be the partition of $n$ where the parts have sizes equal to the numbers of vertices in the connected components of the graph $(V(T), S)$.
	
	\begin{defn} For a tree $T$ and a positive integer $k$, let $P_k(T)$ be the set of edge subsets $S$ partitioning $T$ into connected components each of size at most $k$: \[P_k(T) = \{ S | S \subseteq E(T), \lambda(S) = (a_1, \dots, a_m), a_1 \le a_2 \le \dots \le a_m \le k \} \subseteq \mathcal{P}(E(T)). \] We then define the $k$\textit{-truncated chromatic symmetric function} of $T$, denoted ${}_kX_T$, as follows: \[ {}_kX_T = \sum_{S \in P_k(T)} (-1)^{|S|} p_{\lambda(S)} \] \end{defn}
	
	Note that, by Theorem $2.5$ of \cite{Stanley}, when $k = n$ we obtain the original chromatic symmetric function $X_T$: ${}_kX_T = \sum_{S \subseteq E(T)} (-1)^{|S|} p_{\lambda(S)} = X_T$. 
	
	\begin{rmk} Our approach of considering only some of the possible subsets of edges of $T$ is similar to that used by Smith, Smith, and Tian in \cite{Smith}. There, they consider the terms in the chromatic symmetric function corresponding to subsets $S \subseteq E(T)$ such that $|S| \le k$, namely all $m$-cuts of $T$ where $m \le k$. They showed in \cite{Smith} that computing the terms in the chromatic symmetric function for all such subsets $S$ where $|S| \le 5$ suffices to prove that the chromatic symmetric function distinguishes unrooted trees on at most $24$ vertices.
	
	Here, instead of considering subsets $S \subseteq E(T)$ such that $|S| \le k$ for a constant $k$, we consider subsets $S \subseteq E(T)$ such that any subset $U \subseteq V(T)$ for which $uv \in E(T) \implies uv \in S$ for all $u,v \in U$ satisfies $|U| \le k$. In other words, the subsets we consider partition $T$ into connected components of size at most $k$. \end{rmk}
	
	\begin{lem}\label{lem:truncstrong} Let $T$ and $T'$ be trees on $n$ vertices. Then, for any positive integer $k \le n$, if $X_T = X_{T'}$ then ${}_kX_T = {}_kX_{T'}$. Therefore, if for all pairs $(T,T')$ of non-isomorphic trees on $n$ vertices we have ${}_kX_T \not= {}_k X_{T'}$ for some $k$, then we must also have $X_T \not= X_{T'}$.
	\end{lem}
	\begin{proof}
	There exist coefficients $a_{\lambda}(T) \in \integers$ for each partition $\lambda \vdash n$ such that \[ X_T = \sum_{\lambda \vdash n} a_{\lambda}(T) p_{\lambda}, \] all of which are determined by the chromatic symmetric function $X_T$. For each partition $\lambda \vdash n$, define $\delta_k(\lambda) =1$ if all parts of lambda are at most $k$, and $\delta_k(\lambda) = 0$ otherwise. Then, let $b_{\lambda}(T) = \delta_k(\lambda)a_{\lambda}(T)$. Since $a_{\lambda}(T)$ is determined by $X_T$ for each $\lambda$, $b_{\lambda}(T)$ is also determined by $X_T$. Note that \[{}_kX_T = \sum_{S \in P_k} (-1)^{|S|} p_{\lambda(S)} = \sum_{\lambda \vdash n, \delta_k(\lambda) = 1} a_{\lambda}(T)p_{\lambda} = \sum_{\lambda \vdash n} \delta_{k}(\lambda) a_{\lambda}(T) p_{\lambda} = \sum_{\lambda \vdash n} b_{\lambda}(T) p_{\lambda}, \] so ${}_kX_T$ is a function of $X_T$. 
	
	Therefore, we have that if $X_{T} = X_{T'}$, then ${}_kX_T = {}_kX_{T'}$, as desired.
	\end{proof}
	
	\begin{rmk}
	    For $q \in \nats$ and any $C \in (\ints / q \ints)^n$ such that $C_j = 0$ for all $j > k$, we have that $\varphi_{q,C}({}_kX_T) = \varphi_{q,C}(X_T)$.
	\end{rmk}
	
	Performing Algorithm \ref{alg:computeecsf} using an $n$-tuple $C$ for which all but the first $k$ entries are zero allows us to omit the terms $r_{k+1}, \dots, r_n$ in the intermediate step of computing the $C$-evaluated symmetric function sequence (Algorithm \ref{alg:computeesfs}). Furthermore, since for such an $n$-tuple $C$ each term $r_i$ with $i \ge k$ from a previous recursive call in Algorithm \ref{alg:computeesfs} only affects terms $r_j$, where $j \ge i \ge k$, in future calls, these remaining terms can be omitted in each call to \ref{alg:computeesfs}. By avoiding the computation of these superfluous terms, we can improve the runtime of Algorithm \ref{alg:computeecsf} in this special case from $O(n^2(\log q)^2)$ to $O(nk(\log q)^2)$. 
	
	By Lemma \ref{lem:truncstrong} (and also by Lemma \ref{lem:homom-ver} using $C = (C_1, \dots, C_k, 0, 0, \dots, 0)$) to show that $X_T$ distinguishes all trees on at most $n$ vertices up to isomorphism, it suffices to show that ${}_kX_T$ distinguishes all such trees, for a constant $k$. 
	
	Setting $k = 3$, we obtained the following computational result:
	
	\begin{thm}
	    The $3$-truncated chromatic symmetric function distinguishes all  trees on at most $29$ vertices up to isomorphism. Hence, the chromatic symmetric function distinguishes all non-isomorphic trees on up through $29$ vertices.
	\end{thm}
	
	Using a computer, we generated primes $q_1, q_2, \dots, q_m$ and tuples $C_1, C_2, \dots, C_m$ of the form $C_j = (x_j, y_j, z_j, 0, 0, \dots)$, where $x_j, y_j, z_j \in (\ints / q_j \ints)$. Then, using Keeler Russell's C++ library (which is provided in \cite{Russell}) to generate all non-isomorphic trees on $n$ vertices, we successively computed $\varphi_{q_1,C_1}(T)$ for each tree $T$. For the elements of each equivalence class $S_r$ of trees $T$ such that $\varphi_{q_1,C_1}(T) = r$, we computed $\varphi_{q_2,C_2}(T)$ and generated smaller equivalence classes of trees with each ordered pair of images under the homomorphisms $\varphi_{q_1,C_1}$ and $\varphi_{q_2,C_2}$. We then repeated this process until each equivalence class contained only a single tree, which occurred after a finite sequence of homomorphisms $\varphi_{q_1,C_1}, \dots, \varphi_{q_m,C_m}$ were applied.
	
	The computer program we used to verify that the CSF distinguishes trees on at most $29$ vertices is available at \url{https://github.com/VietaFan/CSFTreeConjecture}. 
	
	This computational result leads to the following question:
	\begin{ques}
	    Does the $3$-truncated chromatic symmetric function distinguish all unrooted trees up to isomorphism? 
	\end{ques}
	
	
	\section{Acknowledgements}
	
	We would like to thank Dr. John Shareshian for many helpful discussions on the theory of chromatic symmetric functions that led to this paper. 
	
	We would also like to thank Keeler Russell for making the code he used to verify that the chromatic symmetric function distinguishes trees on at most 25 vertices freely available on Github. His library for sequentially generating all non-isomorphic free trees on $n$ vertices was particularly useful for our verification that the CSF distinguishes all trees on at most 29 vertices.

\end{document}